\documentclass{amsart}
\usepackage{amsfonts,amssymb,amsmath,euscript}
\usepackage{graphicx}
\usepackage[matrix,arrow,curve]{xy}
\usepackage{wrapfig}

  \usepackage{amssymb}
\usepackage{mathrsfs}
\usepackage{euscript}
\usepackage{lipsum}

    \newtheorem{thm}{Theorem}                     [section]
    \newtheorem{thm*}{Theorem}

    \newtheorem{lemma*}{Lemma}    

\newcommand{\hilb}{\mathcal H}
\newcommand{\scR}{\mathscr R}

\newcommand{\euE}{\EuScript E}

\newcommand{\clK}{\mathcal K}

\newcommand{\mbR}{\mathbb R}

\newcommand{\eps}{\varepsilon}

\begin{document}
\title[Limiting absorption principle for perturbed operator]{Limiting absorption principle\\ for perturbed operator}
\author{Nurulla Azamov}
\address{Independent scholar, Adelaide, SA, Australia}
\email{azamovnurulla@gmail.com}
\begin{abstract} 
In this note the following theorem is proved.
Let $\hilb$ and $\clK$ be Hilbert spaces. 
Let $H_0$ be a self-adjoint operator on $\hilb,$  $F \colon \hilb \to \clK$ be a closed $|H_0|^{1/2}$-compact operator, and $J \colon \clK \to \clK$ be a bounded self-adjoint operator.
If the operator 
$$
    F (H_0 - \lambda -  iy)^{-1} F^*
$$
has norm limit as $y \to 0^+$ for a.e.~$\lambda,$ then so does the operator 
$$
    F (H_0 + F^*JF - \lambda -  iy)^{-1} F^*.
$$
\end{abstract}
\maketitle

\bigskip 

\section{Introduction}
The limiting absorption principle (LAP) is one of the main tools for studying the essential spectrum of a self-adjoint operator.
Roughly speaking, LAP asserts that the resolvent of a self-adjoint operator, properly regularised, converges a.e. in an appropriate topology as the spectral parameter approaches the essential spectrum.
Importance of LAP for spectral, scattering and PDE theories is well-known,  --- it suffices to say that LAP forms a foundation of the stationary approach to scattering theory, see e.g. \cite{Agm, AMG, BW, Kur, YaBook}.

LAP can be formulated in many different ways, depending on a choice of regularisation of the resolvent, the topology of convergence and whether the regularised resolvent is required to converge everywhere or almost everywhere. 
The one which we use is as follows. Let $H_0$ be a self-adjoint operator on a (complex separable) Hilbert space~$\hilb,$ which is endowed with a rigging in the form of 
a closed $|H_0|^{1/2}$-compact operator~$F$ with zero kernel and co-kernel, acting from~$\hilb$ to possibly another Hilbert space~$\clK.$  The pair $(\hilb,F)$ we call a \emph{rigged Hilbert space}. 
We say that~$H_0$ obeys LAP (in the norm topology, unless specified otherwise), if the sandwiched resolvent of $H_0,$ 
\begin{equation} \label{F0}
   T_{ \lambda +  iy}(H_0) :=    F  R_{ \lambda +  iy}(H_0) F^*,
\end{equation}
has the norm limit,    $T_{ \lambda +  i0}(H_0),$ for a.e. $\lambda \in \mbR$ as $y \to 0^+.$

A natural question is whether LAP holds for a perturbed operator of the form $H_0 + F^*JF,$ if it holds for $H_0,$
where $J$ is a bounded self-adjoint operator on $\clK.$  
An affirmative answer to this question is seemingly obvious. Indeed, by the second resolvent identity, we have 
\begin{equation} \label{F2}
   T_{\lambda + iy}(H_r) = T_{\lambda + iy}(H_0) (1 + r J T_{\lambda + iy}(H_0) )^{-1},
   \end{equation}
where $H_r = H_0 + r F^*JF.$ Hence, provided $T_{\lambda + i0}(H_0)$ exists for a given choice of~$\lambda,$
the norm limit    $T_{\lambda + i0}(H_r)$ exists if and only if the operator 
\begin{equation} \label{F1}
   1 + r J T_{\lambda + i0}(H_0)
\end{equation}
is invertible. While this operator can fail to be invertible for some values of $r,$ by analytic Fredholm alternative the set of such values is discrete, and so,
it seems unlikely that the operator (with a particular choice $r=1$ of the coupling variable~$r$)
$$
   1 + J T_{\lambda + i0}(H_0)
$$
may fail to be invertible for a set of~$\lambda$'s of positive Lebesgue measure. However, a formal proof of such statement encounters some difficulties. 
In this paper we provide such a proof. Namely, we prove the following 
\begin{thm}  \label{T1}
Let $H_0$ be a self-adjoint operator acting on a rigged Hilbert space $(\hilb,F),$ as defined above.
If the limiting absorption principle in the norm topology holds for~$H_0,$ then it also holds for any $H_0 + F^*JF,$ where $J$ is a bounded self-adjoint operator on $\clK.$
\end{thm}

\section{Proof}
First I recall some definitions from \cite{AzSFIES}. Let $H_0,$ $F$ and $J$ be as above. 
Let $V = F^*JF,$ and let $z$ be a complex number from the resolvent set of $H_0.$
We say that a complex number $r_z$ is a \emph{(coupling) resonance point} of the pair $H_0,V,$ if the operator 
$$
    1 + r_z T_z(H_0) J
$$
is not invertible. The collection of all resonance points $r_z$ forms a multi-valued holomorphic function, $\scR(z),$  on the resolvent set of $H_0.$ 
The analytic function $\scR(z)$ can have singularities such as branching points. If we assume that the function $\scR(z)$ has no singularities, that is, if it is simply a collection of single-valued holomorphic functions $r_j(z),$
then proof of Theorem \ref{T1}  would be easy. Indeed, in this case each of the coupling resonance functions $r_j(z)$ is either a Herglotz function or a negative of a Herglotz function.
A Herglotz function has limit values $r_j(\lambda+i0)$ for a.e.~$\lambda,$ and the set of~$\lambda$ for which $r_j(\lambda+i0)$ is equal to a pre-given value, say, $1,$ has Lebesgue measure zero
(for these classical properties of Herglotz functions see e.g. \cite{Don, Ho}). 
Since the collection of resonance functions $r_1(z), r_2(z), \ldots$ is countable, each of the operators 
$$
    1 + r_j(\lambda+i0) T_{\lambda+i0}(H_0) J, \ j = 1,2, \ldots, 
$$
is invertible on a common set of~$\lambda$'s of full Lebesgue measure. 
From the argument given in the introduction, see \eqref{F1}, it is clear that $T_{\lambda+i0}(H_0+F^*JF)$ exists if and only if for some resonance point we have $r_{\lambda+i0} = 1.$ 
This completes the proof. 

\smallskip
In the general case this argument needs an adjustment, since $\scR(z)$ can have singularities. 
To overcome them, we use the following theorem, see~\cite[Theorem 2.1]{AzSFIESIII}.

\begin{thm} \label{T0}   Assume that $H_0$ is a self-adjoint operator on a rigged Hilbert space $(\hilb,F),$ which obeys LAP in the norm topology.  Let $I \subset \mbR$  be an open interval.
Then for any $\eps>0$ there exists a compact set $K \subset I$ such that $|I\setminus K| < \eps$ and $K$ has a non-tangential neighbourhood $O(K)$ in the upper and lower half-planes 
with the following property: all coupling resonance functions $r_j(z)$ restricted to $O(K)$ either
\begin{enumerate}
   \item  do not take a value in $[0,1],$ or 
   \item they are single-valued continuous functions without branching points in $O(K),$ nor they have any other kind of singularities.
\end{enumerate}
Moreover, the number of coupling resonances obeying scenario (2) is finite. 
\end{thm}

For the definition of ``non-tangential neighbourhood'' see \cite{AzSFIESIII}. 
By the Riemann mapping theorem, the neighbourhood $O(K)$ can be chosen to be conformally equivalent to the upper half-plane. Therefore, the argument above shows that the set of $\lambda \in K$
where the second scenario coupling resonance functions take value~$1$ has Lebesgue measure zero. Since $\eps>0$ and~$I$ are arbitrary, the proof  of Theorem \ref{T1} is complete. 

\medskip 
As an application of Theorem \ref{T1} we prove the following theorem. For definitions of the wave operators $W_\pm(H_1,H_0)$ and the scattering operator $\mathbf S(H_1,H_0)$ see e.g.~\cite{BW,Kur,YaBook}.
\begin{thm} Let $H_0$ be a self-adjoint operator on a rigged Hilbert space $(\hilb,F).$ 
If LAP in the norm topology holds for $H_0$ then for any bounded self-adjoint operator $J$ on $\clK$ the wave operators 
$
   W_\pm (H_0 + F^*JF, H_0)
$
exist and are complete, and therefore, the scattering operator $\mathbf S(H_0+F^*JF, H_0)$ also exists. 
\end{thm} 
\begin{proof} This is of course a classical theorem of stationary scattering theory, see e.g. \cite{BE, YaBook}, and \cite{AzDaMN} for a recent new proof, but its premise requires LAP to hold for both~$H_0$ and $H_0 + F^*JF.$ 
Theorem \ref{T1} ensures this. 
\end{proof}

\medskip
Theorem \ref{T1} admits an obvious generalisation. Let $\euE$ be an invariant operator ideal (see e.g. \cite{SiTrId2} for definition). 
We say that a self-adjoint operator $H_0$ obeys LAP in $\euE,$ if the sandwiched resolvent $T_z(H_0)$ belongs to $\euE$ and 
the limit \eqref{F0} exists in $\euE$-norm. An invariant operator ideal version of Theorem \ref{T0} also holds, see \cite{AzSFIESIII}. 
It obviously follows from the sandwiched version of the second resolvent identity \eqref{F2} and the ideal property of $\euE$
that if $H_0$ obeys LAP in $\euE$ and the operator \eqref{F1} is invertible, then $H_r$ also obeys LAP in $\euE.$ 
Therefore, precisely the same argument as above proves the following 
\begin{thm}  \label{T3}
Let $H_0$ be a self-adjoint operator acting on a rigged Hilbert space $(\hilb,F),$ as defined above. Suppose $T_z(H_0) $ belongs to an invariant operator ideal $\euE.$
If the limiting absorption principle for~$H_0$ holds in $\euE,$  then it holds in $\euE$ for any $H_0 + F^*JF,$ where $J$ is a bounded self-adjoint operator on $\clK.$
\end{thm}

Since convergence in $\euE$ implies norm convergence, the wave and scattering operators also exist in the setting of Theorem \ref{T3}. But in this case a little more can be said. 
Let $H_r = H_0 + r F^*JF.$ 
The scattering operator $\mathbf S(H_1, H_0)$ commutes with~$H_0$ and admits a direct integral decomposition
$$
     \mathbf S(H_1, H_0)  = \int^\oplus _{\Lambda}  S(\lambda; H_1,H_0) \, d\lambda,
$$
where $\Lambda$ is the full set of real numbers $\lambda$ for which both $T_{\lambda+i0}(H_0)$ and $T_{\lambda+i0}(H_1)$ exist,
and $S(\lambda; H_1,H_0)$ is the scattering matrix on the fiber Hilbert space $\mathfrak h_\lambda,$ see e.g. \cite{YaBook} and \cite{AzDaMN} for more details. 
We can state the following: if $T_z(H_0)$ belongs to $\euE(\clK)$ and LAP for $H_0$ holds in $\euE(\clK)$ then for any $J$ as above the scattering matrix  $S(\lambda; H_1,H_0)$ exists for a.e. $\lambda$ and belongs to~$\euE(\mathfrak h_\lambda).$
In particular, in the special case of the trace class ideal $\euE = \mathcal L_1(\mathcal K),$ the Fredholm determinant $\det S(\lambda; H_1,H_0)$ and the spectral shift function $\xi(\lambda; H_1, H_0)$ are well-defined a.e. and the Birman-Krein formula 
$\det S(\lambda; H_1,H_0) = e^{-2 \pi i \xi(\lambda; H_1, H_0)}$ holds for a.e. $\lambda,$ provided $T_{\lambda+i0}(H_0)$ exists a.e. in the trace class norm. The point here is that all these conclusions can be made without assuming the same about 
$T_{\lambda+i0}(H_1).$

\medskip Further, everything said above clearly applies to the case where the limit $T_{\lambda+i0}(H_0),$ whether in norm or in $\euE,$ exists a.e. only in an open interval~$I.$ 
In this case the wave operators should be replaced by partial wave operators.

Finally, it is possible that although $T_z(H_0)$ belongs to an invariant operator ideal~$\euE,$ the limit $T_{\lambda+i0}(H_0)$ does not exist for a.e. $\lambda$ in the norm of $\euE,$
but the limit of the imaginary part $\Im T_{\lambda + i0}(H_0)$ does. In this case by the same argument one infers that the limit of the imaginary part $\Im T_{\lambda + i0}(H_1)$ 
of the perturbed operator also exists for a.e. $\lambda$ in the norm of $\euE,$ via a well-known formula
$$
   \Im T_{z}(H_1) = (1 + T_{\bar z}(H_0) J )^{-1}  \Im T_{z}(H_0) (1 + J T_{z}(H_0) )^{-1}.
$$
To see this it suffices to recall that if $r_z$ is a resonance point corresponding to $z,$ then $\bar r_z$ is a resonance point corresponding to $\bar z,$
and thus existence of $r_{\lambda+i0}$ is equivalent to the existence of $r_{\lambda-i0}.$ This is not unrelated to scattering theory, since if $\Im T_{\lambda + i0}(H_0)$ belongs to an invariant operator ideal $\euE(\clK),$
then the scattering matrix $S(\lambda; H_1, H_0)$ also belongs to $\euE(\mathfrak h _\lambda).$

\bigskip
{\it Acknowledgements.} I thank my wife for financially supporting me during the work on this paper.

\end{document}